\documentclass[12pt,a4paper,reqno]{amsart}
\usepackage[a4paper,left=30mm,right=30mm,top=36mm,bottom=36mm]{geometry}
\usepackage{amssymb}
\usepackage{graphicx}
\usepackage{epsfig}
\usepackage{nicefrac}
\usepackage{mathrsfs}
\usepackage{float}
\usepackage[utf8]{inputenc}
\usepackage[english]{babel}
\usepackage{pdfpages}
\usepackage{bbm}
\usepackage{esint}
\usepackage{blindtext}
\usepackage{subfigure}
\newcommand{\R}{\mathbb{R}}

\newcommand{\Z}{\mathbb{Z}}
\newcommand{\eps}{\varepsilon}
\newcommand{\fhi}{\varphi}

\renewcommand{\div}{\mathrm{div}}

\def\calO{\mathcal{O}}

\def\calS{\mathcal{S}}

\newtheorem{remark}{Remark}[section]    
\newtheorem{theorem}{Theorem}[section]    
   
\newtheorem{lemma}{Lemma}[section]	
\newtheorem{corollary}{Corollary}[section]   
\numberwithin{equation}{section}

\begin{document}

\title[Optimal rates in nondivergence-form homogenization]{Optimal convergence rates for elliptic homogenization problems in nondivergence-form: analysis and numerical illustrations}

\author[T. Sprekeler]{Timo Sprekeler}
\address[Timo Sprekeler]{University of Oxford, Mathematical Institute, Woodstock Road, Oxford OX2 6GG, UK.}
\email{sprekeler@maths.ox.ac.uk}

\author[H. V. Tran]{Hung V. Tran}
\address[Hung V. Tran]{Department of Mathematics, University of Wisconsin Madison, Van Vleck hall, 480 Lincoln drive, Madison, WI 53706, USA}
\email{hung@math.wisc.edu}

\date{\today}

\begin{abstract}
We study optimal convergence rates in the periodic homogenization of linear elliptic equations of the form $-A(x/\eps):D^2 u^{\eps} = f$ subject to a homogeneous Dirichlet boundary condition. 
We show that the optimal rate for the convergence of $u^{\eps}$ to the solution of the corresponding homogenized problem in the $W^{1,p}$-norm is $\calO(\eps)$. We further obtain optimal gradient and Hessian bounds with correction terms taken into account in the $L^p$-norm. 
We then provide an explicit $c$-bad diffusion matrix and use it to perform various numerical experiments, which demonstrate the optimality of the obtained rates.
\end{abstract}

\maketitle

\section{Introduction}

In this work, we study optimal rates in the periodic homogenization of elliptic equations in nondivergence-form. We consider the linear prototype equation subject to a homogeneous Dirichlet boundary condition, posed on a bounded smooth domain $\Omega\subset\R^n$, i.e., problems of the form
\begin{align}\label{ueps problem}
\left\{\begin{aligned}-A\left(\frac{\cdot}{\eps}\right):D^2 u^{\eps} &= f& &\text{in }\Omega,\\
u^{\eps} &= 0&  &\text{on }\partial\Omega,\end{aligned}\right.
\end{align}
with a parameter $\eps>0$ (considered to be small), a right-hand side 
\begin{align*}
f\in W^{3,q}(\Omega)\quad\text{for some }q>n,
\end{align*}
and a symmetric, $\Z^n$-periodic, uniformly elliptic coefficient function
\begin{align*}
A\in C^{0,\alpha}(\mathbbm{T}^n;\calS^n_{+})\quad\text{for some }\alpha\in (0,1].
\end{align*}
Here, $\mathbbm{T}^n:=\R^n/\Z^n$ denotes the flat $n$-dimensional torus and $\calS^n_{+}\subset \R^{n\times n}$ the set of symmetric positive definite $n\times n$ matrices. Throughout this work, we denote the unit cell in $\R^n$ by
\begin{align*}
Y:=[0,1]^n\subset\R^n.
\end{align*}

In the theory of periodic homogenization, it is well-known (see e.g., Bensoussan, Lions, Papanicolaou \cite{BLP11}, Jikov, Kozlov, Oleinik \cite{JKO94}) that as the parameter $\eps$ tends to zero, the corresponding sequence of solutions $(u^{\eps})_{\eps>0}$ to \eqref{ueps problem} converges uniformly on $\bar{\Omega}$ to the solution $u$ of the homogenized problem
\begin{align}\label{u problem}
\left\{\begin{aligned}-\bar{A}:D^2 u &= f& &\text{in }\Omega,\\
u &= 0& &\text{on }\partial\Omega.\end{aligned}\right.
\end{align}
Here, the effective coefficient $\bar{A}\in\calS^{n}_{+}$ is a constant positive definite matrix, and can be obtained through integration against an invariant measure, that is
\begin{align*}
\bar{A}:= \int_Y Ar 
\end{align*}
with the invariant measure $r:\R^n\rightarrow\R$ being the solution to the periodic problem
\begin{align*}
-D^2:(Ar) = 0\quad\text{in }Y,\qquad r \text{ is }Y\text{-periodic},\qquad r>0,\qquad \int_Y r = 1,
\end{align*}
see e.g., Avellaneda, Lin \cite{AL89}, Engquist, Souganidis \cite{ES08}. The effective coefficient $\bar{A}$ can be equivalently characterized via corrector functions: For $i,j\in\{1,\dots,n\}$, the $(i,j)$-th entry $\bar{a}_{ij}$ of $\bar{A}$ is the unique value such that the periodic cell problem
\begin{align}\label{vij problem}
-A:D^2 v^{ij} = a_{ij}-\bar{a}_{ij}\quad\text{in }Y,\qquad v^{ij} \text{ is }Y\text{-periodic},\qquad \int_Y v^{ij} = 0
\end{align}
admits a unique solution $v^{ij}:\R^n\rightarrow \R$, called a corrector function. 

We are interested in optimal rates for the convergence of the solution $u^{\eps}$ of \eqref{ueps problem} to the solution $u$ of the homogenized problem \eqref{u problem} in appropriate function spaces. Optimal rates in $L^{\infty}(\Omega)$ have recently been obtained in Guo, Tran, Yu \cite{GTY20}. With $c_j^{kl}\in\R$, $j,k,l\in\{1,\dots,n\}$, defined by
\begin{align}\label{cjkl def}
c_j^{kl}=c_j^{kl}(A):=\int_Y Ae_j\cdot \nabla v^{kl}\, r,
\end{align}
the function $h$ defined by
\begin{align}\label{h definition}
h:=\sum_{j,k,l=1}^n c_j^{kl}\partial^3_{jkl}u,
\end{align}
and the solution $z$ to the problem
\begin{align}\label{z problem}
\left\{\begin{aligned}-\bar{A}:D^2 z &= -h&  &\text{in }\Omega,\\
z &= 0&  &\text{on }\partial\Omega,\end{aligned}\right.
\end{align}
the main result in \cite{GTY20} states the following: 
\begin{theorem}[Theorem 1.2 in \cite{GTY20}]\label{Theorem Guo-Tran-Yu}
Assume that $A\in C^2(\mathbbm{T}^n;\calS^n_{+})$ and $f\in C^3(\bar{\Omega})$. Let $u^{\eps}$, $u$ and $z$ be the solutions to \eqref{ueps problem}, \eqref{u problem} and \eqref{z problem} respectively. Then we have
\begin{align*}
\|u^{\eps}-u+2\eps z\|_{L^{\infty}(\Omega)}=\calO(\eps^2).
\end{align*}
In particular, with $h$ given by \eqref{h definition}, the following assertions hold:
\begin{itemize}
\item[(i)] If $h\equiv 0$, then $\|u^{\eps}-u\|_{L^{\infty}(\Omega)}=\calO(\eps^2)$ and this rate of convergence is optimal.
\item[(ii)] If $h\not\equiv 0$, then $\|u^{\eps}-u\|_{L^{\infty}(\Omega)}=\calO(\eps)$ and this rate of convergence is optimal.
\end{itemize}
\end{theorem}
\begin{remark}
There is a typo in \cite{GTY20}, which uses the opposite sign for the $\calO(\eps)$-term.
\end{remark}

Let us recall the $\calO$ notation, which we are going to use throughout this paper: For a function  $e: (0,\infty)\ni \eps\mapsto e(\eps)\in [0,\infty)$ and an exponent $\gamma\geq 0$, we write
\begin{align*}
e(\eps)=\calO(\eps^{\gamma}) \quad \Longleftrightarrow \quad \exists\, C,\eps_0>0:\; e(\eps)\leq C\eps^{\gamma}\quad\forall \eps \in (0,\eps_0).
\end{align*}
As a consequence of Theorem \ref{Theorem Guo-Tran-Yu}, we can classify coefficients $A\in C^2(\mathbbm{T}^n;\calS^n_{+})$ into those that give optimal rate of convergence $\calO(\eps^2)$, called the $c$-good coefficients, and those that give optimal rate of convergence $\calO(\eps)$, called the $c$-bad coefficients.

\begin{corollary}[$c$-good and $c$-bad matrices]\label{Cor c-good} Let $A\in C^2(\mathbbm{T}^n;\calS^n_{+})$. Then, with $\{c_j^{kl}\}_{1\leq j,k,l\leq n}$ given by \eqref{cjkl def}, the following assertions hold: 
\begin{itemize}
\item[(i)] If $c_j^{kl}(A)=0$ for all $j,k,l\in\{1,\dots,n\}$, then  the situation (i) of Theorem \ref{Theorem Guo-Tran-Yu} occurs for any choice of $f$. We then say $A$ is $c$-good.
\item[(ii)] If $c_j^{kl}(A)\neq 0$ for some $j,k,l\in\{1,\dots,n\}$, then there exists $f$ such that the situation (ii) of Theorem \ref{Theorem Guo-Tran-Yu} occurs. We then say $A$ is $c$-bad.
\end{itemize}
\end{corollary}

It has further been shown that the set of $c$-bad matrices is open and dense in $C^2(\mathbbm{T}^n;\calS^n_{+})$ for dimensions $n\geq 2$ (see Theorem 1.4 in \cite{GTY20}). Therefore, we have generically that the optimal rate is $\calO(\eps)$ in $L^{\infty}(\Omega)$. Related results on convergence rates and error estimates in the periodic homogenization of elliptic equations in divergence-form have been derived by various authors; see e.g., \cite{Gri06,KLS12,MV97,OV07,Sus13} and the references therein.

The objective of this work are the optimal rates in higher-order norms. This question has not been studied yet and it seems that the only available result in higher-order norms is the following $W^{2,p}$ corrector estimate from Capdeboscq, Sprekeler, S\"{u}li \cite{CSS20}:

\begin{theorem}[Theorem 2.8 in \cite{CSS20}]\label{Capdeboscq-Sprekeler-Suli}
Assume that $A\in C^{0,\alpha}(\mathbbm{T}^n;\calS^n_{+})$ for some $\alpha\in(0,1]$ and $f\in W^{2,p}(\Omega)$ for some $p\in (1,\infty)$. Further, assume that the solution $u$ to \eqref{u problem} satisfies $u\in W^{4,p}(\Omega)\cap W^{2,\infty}(\Omega)$. Then, with $u^{\eps}$ given by \eqref{ueps problem} and $V=(v^{ij})_{1\leq i,j\leq n}$ given by \eqref{vij problem}, we have
\begin{align*}
\left\|u^{\eps}- u - \eps^2\,  V\left(\frac{\cdot}{\eps}\right):D^2 u  \right\|_{W^{2,p}(\Omega)}=\calO(\eps^{\frac{1}{p}}).
\end{align*}
\end{theorem}

Note that the standing assumption in \cite{CSS20} is $A\in W^{1,q}(Y)\cap C^{0,\alpha}(\mathbbm{T}^n;\calS^n_{+})$ for some $q>n$, which is useful for the numerical homogenization but not essential for the result of Theorem \ref{Capdeboscq-Sprekeler-Suli}. We observe that we cannot expect strong convergence of $u^{\eps}$ to the homogenized solution $u$ in $W^{2,p}(\Omega)$ and that it is necessary to add corrector terms. The optimality of the rate of convergence $\calO(\eps^{\frac{1}{p}})$ in Theorem \ref{Capdeboscq-Sprekeler-Suli} has not been discussed yet, which is a gap of knowledge we want to fill. The main contribution of this work is to derive optimal $W^{1,p}(\Omega)$ estimates for $p\in (1,\infty)$ and to provide numerical illustrations.
 
For the numerical homogenization of linear equations in nondivergence-form, we refer the reader to Capdeboscq, Sprekeler, S\"{u}li \cite{CSS20}, Froese, Oberman \cite{FO09}, and the references therein. Let us note that the divergence-form case was the focus of active research over the past decades; see e.g., the works \cite{AEE12,EE03,EH09,EW02,HW97} by various authors on heterogeneous multiscale methods and multiscale finite element methods. 

For some results on fully nonlinear equations of nondivergence-structure, we refer to Camilli, Marchi \cite{CM09}, Kim, Lee \cite{KL16} for convergence rates and to Gallistl, Sprekeler, S\"{u}li \cite{GSS20}, Finlay, Oberman \cite{FO18} for numerical homogenization of Hamilton--Jacobi--Bellman equations.

\subsection{Main results}

The main result is the following theorem on optimal rates for the convergence of $u^{\eps}$ to the homogenized solution $u$ in $W^{1,p}(\Omega)$:

\begin{theorem}[$W^{1,p}$ estimate and optimal rate]\label{Main theorem}
Assume that $A\in C^{0,\alpha}(\mathbbm{T}^n;\calS^n_{+})$ for some $\alpha\in (0,1]$ and $f\in W^{3,q}(\Omega)$ for some $q>n$. Let $u^{\eps}$, $u$ and $z$ be the solutions to \eqref{ueps problem}, \eqref{u problem} and \eqref{z problem} respectively. Further, let $V=(v^{ij})_{1\leq i,j\leq n}$ be the matrix of corrector functions given by \eqref{vij problem}. Then, for all $p\in (1,\infty)$, we have that
\begin{align*}
\left\|u^{\eps}- u +2\eps z - \eps^2\,  V\left(\frac{\cdot}{\eps}\right):D^2 u  \right\|_{W^{1,p}(\Omega)}=\calO(\eps^{1+\frac{1}{p}}).
\end{align*}
In particular, for all $p\in (1,\infty)$, the sequence $(u^{\eps})_{\eps>0}$ converges to the homogenized solution $u$ strongly in $W^{1,p}(\Omega)$ with the rate
\begin{align*}
\left\|u^{\eps}- u \right\|_{W^{1,p}(\Omega)}=\calO(\eps),
\end{align*}
and this rate of convergence $\calO(\eps)$ is optimal in general.
\end{theorem}

\begin{remark}[$L^{\infty}$ estimate, gradient estimate and Hessian estimate]\label{rmrk some estimates}
In the situation of Theorem \ref{Main theorem}, the following assertions hold.
\begin{itemize}
\item[(i)] $L^{\infty}$ bound: An inspection of the proof, see \eqref{a W1infty estimate} and \eqref{zeps to z rate}, yields that
\begin{align}\label{Guo-Tran-Yu}
\|u^{\eps}-u+2\eps z\|_{L^{\infty}(\Omega)}=\calO(\eps^2).
\end{align}
Hence, we recover the result on optimal $L^{\infty}$ rates from Theorem \ref{Theorem Guo-Tran-Yu} under these weaker assumptions on the coefficient $A$ and the right-hand side $f$. 
\item[(ii)] Gradient bound: For all $p\in (1,\infty)$, we have
\begin{align}\label{gradestimate}
\left\|\nabla u^{\eps}- \nabla u +2\eps  \nabla z - \eps \sum_{i,j=1}^n  \nabla v^{ij}\left(\frac{\cdot}{\eps}\right)\partial^2_{ij}u  \right\|_{L^p(\Omega)}=\calO(\eps^{1+\frac{1}{p}}).
\end{align}
\item[(iii)] Hessian bound: In view of Theorem \ref{Capdeboscq-Sprekeler-Suli}, for all $p\in (1,\infty)$, there holds
\begin{align}\label{hessianestimate}
\left\|D^2 u^{\eps}- D^2 u -\sum_{i,j=1}^n  D^2 v^{ij}\left(\frac{\cdot}{\eps}\right)\partial^2_{ij}u  \right\|_{L^p(\Omega)}=\calO(\eps^{\frac{1}{p}}).
\end{align}
\end{itemize}
\end{remark}

An essential role in the proof plays the boundary corrector $\theta^{\eps}$, which is defined to be the solution to the following problem with oscillations in the boundary data:
\begin{align}\label{thetaeps problem}
\left\{\begin{aligned}-A\left(\frac{\cdot}{\eps}\right):D^2 \theta^{\eps} &= 0& &\text{in }\Omega,\\
\theta^{\eps} &= -V\left(\frac{\cdot}{\eps}\right):D^2 u&  &\text{on }\partial\Omega.\end{aligned}\right.
\end{align}

We then have the following result on the asymptotic behavior of the boundary corrector under the reduced regularity $f\in W^{2,q}(\Omega)$ for some $q>n$:

\begin{lemma}[Boundary corrector $W^{1,p}$ bound]\label{lemma boundary corrector}
Assume that $A\in C^{0,\alpha}(\mathbbm{T}^n;\calS^n_{+})$ for some $\alpha\in (0,1]$ and $f\in W^{2,q}(\Omega)$ for some $q>n$. Further, let $\theta^{\eps}$ be the solution to the problem  \eqref{thetaeps problem}. Then, for all $p\in (1,\infty)$, we have that
\begin{align}\label{thetaeps estimate}
\eps\left\| \theta^{\eps}\right\|_{W^{1,p}(\Omega)}=\calO(\eps^{\frac{1}{p}}).
\end{align}
\end{lemma}
\begin{remark}[Boundary corrector $W^{2,p}$ bound \cite{CSS20}]
In the situation of Theorem \ref{Capdeboscq-Sprekeler-Suli}, there holds 
\begin{align}\label{thetaeps hessestimate}
\eps^2\left\|\theta^{\eps}\right\|_{W^{2,p}(\Omega)}=\calO(\eps^{\frac{1}{p}}).
\end{align}
\end{remark}

Let us remark that the estimate \eqref{thetaeps estimate} for $p=2$ has been shown in \cite{AA99,OV12} in the context of divergence-form homogenization by energy estimates. 
It is worth noting here that we only obtain $W^{1,p}$ and $W^{2,p}$ bounds for the boundary corrector $\theta^\eps$, and we do not study qualitative and quantitative homogenization of \eqref{thetaeps problem} (for the latter see e.g., \cite{AKM17,FK17,GM12}). 

Another important ingredient in the proof of Theorem \ref{Main theorem} is the proof of the $\calO(\eps)$ rate for the convergence to the homogenized solution in $W^{1,p}(\Omega)$ under the reduced regularity $f\in W^{2,q}(\Omega)$ for some $q>n$:
\begin{lemma}[$W^{1,p}$ convergence rate $\calO(\eps)$]\label{Lemma rate eps}
Assume that $A\in C^{0,\alpha}(\mathbbm{T}^n;\calS^n_{+})$ for some $\alpha\in (0,1]$ and $f\in W^{2,q}(\Omega)$ for some $q>n$. Let $u^{\eps}$ and $u$ be the solutions to \eqref{ueps problem} and \eqref{u problem} respectively. Then, for any $p\in (1,\infty)$, we have
\begin{align*}
\|u^{\eps}- u \|_{W^{1,p}(\Omega)}=\calO(\eps).
\end{align*}
\end{lemma}

Finally, we demonstrate through numerical experiments that the obtained rates in the previously stated results cannot be improved in general. 

\begin{remark}[Optimality of rates]
The rate $\calO(\eps^{2})$ in the $L^{\infty}(\Omega)$ estimate \eqref{Guo-Tran-Yu}, the rate $\calO(\eps^{1+\frac{1}{p}})$ in the gradient estimate \eqref{gradestimate} and the rate $\calO(\eps^{\frac{1}{p}})$ in the Hessian estimate \eqref{hessianestimate} are optimal in general. Consequently, also the rates in the boundary corrector estimates \eqref{thetaeps estimate} and \eqref{thetaeps hessestimate} are optimal in general.
\end{remark}

For the numerical illustrations we use an explicit $c$-bad matrix (recall Corollary \ref{Cor c-good} for the definition of $c$-bad) and consider a homogenization problem of the form \eqref{ueps problem} with $z\not\equiv 0$. This is the first direct proof of the existence of a $c$-bad matrix.  

\begin{theorem}[Explicit $c$-bad matrix]\label{thm explicit cbad}
The matrix-valued function $A:\R^2\rightarrow \R^{2\times 2}$ given by
\begin{align*}
A(y_1,y_2):=\frac{1}{r(y_1,y_2)}\begin{pmatrix}
1-\frac{1}{2}\sin(2\pi y_1)\sin(2\pi y_2) & 0\\0 & 1+\frac{1}{2}\sin(2\pi y_1)\sin(2\pi y_2)
\end{pmatrix}
\end{align*}
with $r:\R^2\rightarrow\R$ defined by
\begin{align*}
r(y_1,y_2):= 1+\frac{1}{4}(\cos(2\pi y_1)-2\sin(2\pi y_1))\sin(2\pi y_2)
\end{align*}
is $c$-bad. More precisely, there holds $c_1^{11}=c_1^{22}=-\frac{1}{128\pi}$ and $c_j^{kl}=0$ otherwise.
\end{theorem}

We briefly explain the organization of the paper:

\subsection{Structure of the paper}

In Section \ref{Sec Proofs}, we prove the main result, i.e., Theorem \ref{Main theorem}. We start by recalling some uniform estimates from the theory of homogenization in Section \ref{subsec uniform est}. Thereafter, we prove Lemmata \ref{lemma boundary corrector} and \ref{Lemma rate eps} in Sections \ref{Sec proof of lemma1} and \ref{Sec proof of lemma2} respectively, and finally the main theorem in Section \ref{Sec Pf of Thm}.

In Section \ref{Sec Num exp}, we provide numerical illustrations of the convergence rates from Remark \ref{rmrk some estimates}. We start by proving Theorem \ref{thm explicit cbad} in Section \ref{Sec expl cbad}, providing an explicit $c$-bad matrix which we use for the numerical experiments. We illustrate the $L^{\infty}$ bound from Remark \ref{rmrk some estimates} (i) in Section \ref{Sec Num ill 1}, the gradient bound from Remark \ref{rmrk some estimates} (ii) in Section \ref{Sec Num ill 2} and the Hessian bound from Remark \ref{rmrk some estimates} (iii) in Section \ref{Sec Num ill 3}. Numerical illustrations comparing $c$-bad and $c$-good problems are provided in Section \ref{Sec Num ill 4}.

Finally in Section \ref{Sec Extensions and concl}, we discuss some extensions to nonsmooth domains and give some concluding remarks.

\section{Proofs of the main results}\label{Sec Proofs}

\subsection{Uniform estimates}\label{subsec uniform est}

Uniform estimates are essential in the theory of homogenization and form the basis for the proofs of the main results. The crucial uniform estimate for the proofs is the uniform $C^{1,\alpha}$ estimate from \cite{AL89} for nondivergence-form homogenization problems.

\begin{lemma}[Theorem 1 in \cite{AL89}]\label{C1alpha est}
Let $\Omega\subset\R^n$ be a bounded $C^{1,\gamma}$ domain. Assume that $A\in C^{0,\alpha}(\mathbbm{T}^n;\calS^n_{+})$ for some $\alpha\in (0,1]$ and $f\in L^q(\Omega)$ for some $q>n$. For $\eps>0$, let $u^{\eps}$ be the solution to the problem \eqref{ueps problem}. Then there exists $\nu\in (0,1]$ such that there holds
\begin{align*}
\|u^{\eps}\|_{C^{1,\nu}(\Omega)}\leq C\|f\|_{L^q(\Omega)}
\end{align*} 
with a constant $C>0$ independent of $\eps$.
\end{lemma}

For the proof of Lemma \ref{lemma boundary corrector} it turns out to be useful to transform the problem \eqref{thetaeps problem} into divergence-form and use the uniform $W^{1,p}$ estimate from \cite{AL91} for divergence-form homogenization problems.

\begin{lemma}[Theorem C in \cite{AL91}]\label{W1p est for divform}
Let $\Omega\subset\R^n$ be a bounded $C^{2,\gamma}$ domain. Assume that $A^{\mathrm{div}}\in C^{0,\alpha}(\mathbbm{T}^n;\R^{n\times n})$ for some $\alpha\in (0,1]$ is a uniformly elliptic coefficient, $F\in L^p(\Omega)$ and $g\in W^{1,p}(\Omega)$ for some $p\in (1,\infty)$. For $\eps\in (0,1]$, let $\rho^{\eps}\in W^{1,p}(\Omega)$ be the solution to the problem
\begin{align*}
\left\{\begin{aligned}-\nabla\cdot \left(A^{\mathrm{div}}\left(\frac{\cdot}{\eps}\right)\nabla \rho^{\eps}\right) &= -\nabla\cdot F& &\text{in }\Omega,\\
\rho^{\eps} &= g&  &\text{on }\partial\Omega.\end{aligned}\right.
\end{align*}
Then we have the estimate
\begin{align*}
\|\rho^{\eps}\|_{W^{1,p}(\Omega)}\leq C\left(\|F\|_{L^p(\Omega)}+\|g\|_{W^{1,p}(\Omega)}\right)
\end{align*}
with a constant $C>0$ independent of $\eps$.
\end{lemma}

With the uniform estimates at hand, we can prove the main results. We start with the proof of Lemma \ref{lemma boundary corrector}.

\subsection{Proof of Lemma \ref{lemma boundary corrector}}\label{Sec proof of lemma1}

The main ingredient for the proof of Theorem \ref{Main theorem} is the asymptotic behavior of the boundary corrector, i.e., the solution $\theta^{\eps}$ to the problem \eqref{thetaeps problem}. We start by proving Lemma \ref{lemma boundary corrector} and it turns out to be useful to transform the problem \eqref{thetaeps problem} into the divergence-form problem
\begin{align*}
\left\{\begin{aligned}-\nabla\cdot \left(A^{\mathrm{div}}\left(\frac{\cdot}{\eps}\right)\nabla\theta^{\eps}\right) &= 0& &\text{in }\Omega,\\
\theta^{\eps} &= -V\left(\frac{\cdot}{\eps}\right):D^2 u&  &\text{on }\partial\Omega,\end{aligned}\right.
\end{align*}
with a coefficient $A^{\mathrm{div}}\in C^{0,\alpha}(\mathbbm{T}^n;\R^{n\times n})$ for some $\alpha\in (0,1]$ that is uniformly elliptic. Indeed, this is a well-known reduction procedure and can be achieved by multiplication of the equation \eqref{thetaeps problem} with the invariant measure and addition of a suitable skew-symmetric matrix; see \cite{AL89}.

\begin{proof}[Proof of Lemma \ref{lemma boundary corrector}]
Firstly note that, as $f\in W^{2,q}(\Omega)$ for some $q>n$, we have $u\in W^{4,q}(\Omega)$ for some $q>n$ and hence also $u\in W^{3,\infty}(\Omega)$. We further note that, as $A\in C^{0,\alpha}(\R^n)$ for some $\alpha\in (0,1]$, we have $V\in C^{2,\alpha}(\R^n)$ by elliptic regularity theory \cite{GT01}. We need to show that 
\begin{align}\label{thetaeps gradestimate}
\eps\| \theta^{\eps}\|_{W^{1,p}(\Omega)} = \calO(\eps^{\frac{1}{p}})
\end{align}
for any $p\in (1,\infty)$. To this end, we let $\eta\in C_c^{\infty}(\R^n)$ be a cut-off function with the properties $0\leq \eta\leq 1$,
\begin{align*}
\begin{aligned}
\eta &\equiv 1& &\text{in}&\left\{x\in \Omega:\mathrm{dist}(x,\partial\Omega)<\frac{\eps}{2}\right\},\\ \eta &\equiv 0& &\text{in}&\left\{x\in \Omega:\mathrm{dist}(x,\partial\Omega)\geq \eps \right\},
\end{aligned}
\end{align*}
and $\lvert \nabla \eta\rvert =\calO(\eps^{-1})$. Note that this implies that
\begin{align}\label{eta asym behav}
\left\| \eta\right\|_{L^p(\Omega)}+\eps \left\| \nabla\eta\right\|_{L^p(\Omega)}=\calO(\eps^{\frac{1}{p}})
\end{align}
for any $p\in (1,\infty)$. We then define the function
\begin{align*}
\tilde{\theta}^{\eps} := \theta^{\eps}+\eta \,V\left(\frac{\cdot}{\eps}\right):D^2 u
\end{align*}
and note that it is the solution to the problem
\begin{align*}
\left\{\begin{aligned}-\nabla\cdot \left(A^{\mathrm{div}}\left(\frac{\cdot}{\eps}\right)\nabla\tilde{\theta}^{\eps}\right) &= -\nabla \cdot F^{\eps}_1& &\text{in }\Omega,\\
\tilde{\theta}^{\eps} &= 0&  &\text{on }\partial\Omega,\end{aligned}\right.
\end{align*}
with $F^{\eps}_1$ given by
\begin{align*}
F^{\eps}_1:= A^{\mathrm{div}}\left(\frac{\cdot}{\eps}\right)\nabla\left[\eta \,V\left(\frac{\cdot}{\eps}\right):D^2 u \right].
\end{align*}
Using the uniform $W^{1,p}$ estimate from Lemma \ref{W1p est for divform}, we find that for $\eps\in (0,1]$ and any $p\in (1,\infty)$, we have
\begin{align*}
\|\tilde{\theta}^{\eps}\|_{W^{1,p}(\Omega)}\leq C  \| F^{\eps}_1\|_{L^p(\Omega)}\leq C  \left\| \eta \,V\left(\frac{\cdot}{\eps}\right):D^2 u  \right\|_{W^{1,p}(\Omega)}.
\end{align*}
Therefore, by the triangle inequality, we obtain the estimate
\begin{align}\label{eps norm grad thetaeps}
\|\theta^{\eps}\|_{W^{1,p}(\Omega)}\leq C \left\| \eta \,V\left(\frac{\cdot}{\eps}\right):D^2 u  \right\|_{W^{1,p}(\Omega)}.
\end{align}
As we have the bound
\begin{align*}
\left\|V\left(\frac{\cdot}{\eps}\right):D^2 u  \right\|_{L^{\infty}(\Omega)}+ \eps \left\| \nabla\left[V\left(\frac{\cdot}{\eps}\right):D^2 u \right] \right\|_{L^{\infty}(\Omega)}=\calO(1),
\end{align*}
and the asymptotic behavior of the cut-off \eqref{eta asym behav}, we deduce from \eqref{eps norm grad thetaeps} that there holds
\begin{align*}
\eps\|\theta^{\eps}\|_{W^{1,p}(\Omega)}&\leq C\eps\left(\left\| \nabla\eta\right\|_{L^p(\Omega)}+ \left(\eps^{-1}+1\right) \left\| \eta\right\|_{L^p(\Omega)}  \right) =\calO(\eps^{\frac{1}{p}}),
\end{align*}
which is precisely the claimed bound \eqref{thetaeps gradestimate}.
\end{proof}

\subsection{Proof of Lemma \ref{Lemma rate eps}}\label{Sec proof of lemma2}

The second ingredient in the proof of Theorem \ref{Main theorem} is the proof of the $\calO(\eps)$ rate for the convergence of $u^{\eps}$ to $u$ under the reduced regularity assumption $f\in W^{2,q}(\Omega)$ for some $q>n$.

\begin{proof}[Proof of Lemma \ref{Lemma rate eps}]
We need to show that for any $p\in (1,\infty)$, there holds
\begin{align}\label{claimed}
\|u^{\eps}- u \|_{W^{1,p}(\Omega)}=\calO(\eps).
\end{align}
With the corrector matrix $V=(v^{ij})_{1\leq i,j\leq n}$ given by \eqref{vij problem} and the boundary corrector $\theta^{\eps}$ given by \eqref{thetaeps problem}, we let
\begin{align}\label{phieps}
\phi^{\eps}:= \eps^2\left[V\left(\frac{\cdot}{\eps}\right):D^2 u+\theta^{\eps}\right].
\end{align}
Then we have that the function $u^{\eps}-u-\phi^{\eps}$ satisfies the problem
\begin{align*}
\left\{\begin{aligned}-A\left(\frac{\cdot}{\eps}\right):D^2 (u^{\eps}-u-\phi^{\eps}) &= \eps F^{\eps}_2& &\text{in }\Omega,\\
u^{\eps}-u-\phi^{\eps} &= 0&  &\text{on }\partial\Omega,\end{aligned}\right.
\end{align*}
with $F^{\eps}_2$ given by
\begin{align*}
F^{\eps}_2:=\sum_{i,j,k,l=1}^n a_{ij}\left(\frac{\cdot}{\eps}\right) \left[2 \partial_i v^{kl}\left(\frac{\cdot}{\eps}\right)\partial^3_{jkl}u+\eps v^{kl}\left(\frac{\cdot}{\eps}\right)\partial^4_{ijkl}u\right].
\end{align*}
As $f\in W^{2,q}(\Omega)$ for some $q>n$, we have $u\in W^{4,q}(\Omega)$ and hence $F^{\eps}_2$ is uniformly bounded in $L^q(\Omega)$. By the uniform estimate from Lemma \ref{C1alpha est}, we have that
\begin{align*}
\|u^{\eps}-u-\phi^{\eps}\|_{W^{1,\infty}(\Omega)}\leq C \eps \|F^{\eps}_2\|_{L^q(\Omega)}=\calO(\eps).
\end{align*}
Finally, by the triangle inequality and Lemma \ref{lemma boundary corrector}, we can conclude that
\begin{align*}
\|u^{\eps}- u \|_{W^{1,p}(\Omega)}\leq \|u^{\eps}-u-\phi^{\eps}\|_{W^{1,p}(\Omega)}+\|\phi^{\eps}\|_{W^{1,p}(\Omega)}=\calO(\eps),
\end{align*}
which is precisely the claimed convergence rate \eqref{claimed}.
\end{proof}

\subsection{Proof of Theorem \ref{Main theorem}}\label{Sec Pf of Thm}

For the proof of the theorem, let us introduce the function $z^{\eps}$ to be the solution to the problem
\begin{align*}
\left\{\begin{aligned}-A\left(\frac{\cdot}{\eps}\right):D^2 z^{\eps} &= -h & &\text{in }\Omega,\\
z^{\eps} &= 0&  &\text{on }\partial\Omega,\end{aligned}\right.
\end{align*}
with the function $h$ defined by \eqref{h definition}. Observe that the function $z$ given by \eqref{z problem} is precisely the homogenized solution corresponding to $(z^{\eps})_{\eps>0}$. We note that as $f\in W^{3,q}(\Omega)$ for some $q>n$, we have $u\in W^{5,q}(\Omega)$ and hence $h\in W^{2,q}(\Omega)$. Therefore, we can apply Lemma \ref{Lemma rate eps} to find that for any $p\in (1,\infty)$, there holds
\begin{align}\label{zeps to z rate}
\|z^{\eps}- z \|_{W^{1,p}(\Omega)}=\calO(\eps).
\end{align}
We further introduce the functions $\chi^{jkl}$, $j,k,l\in \{1,\dots,n\}$ to be the solutions to the periodic problems
\begin{align}\label{chijkl}
-A:D^2 \chi^{jkl}=Ae_j\cdot\nabla v^{kl}-c_j^{kl}\quad\text{in }Y,\quad \chi^{jkl} \text{ is } Y\text{-periodic},\quad \int_Y \chi^{jkl}=0.
\end{align}
Note that the functions $\chi^{jkl}$ are well-defined as by definition \eqref{cjkl def} of $c_j^{kl}$, the right-hand side integrated against the invariant measure equals zero, i.e., there holds
\begin{align*}
\int_Y \left(Ae_j\cdot\nabla v^{kl}-c_j^{kl}  \right)r = 0.
\end{align*}
We also introduce a corresponding boundary corrector $\theta_{\chi}^{\eps}$ to be the solution to the following problem:
\begin{align}\label{thetaepschi problem}
\left\{\begin{aligned}-A\left(\frac{\cdot}{\eps}\right):D^2 \theta_{\chi}^{\eps} &= 0 & &\text{in }\Omega,\\
\theta_{\chi}^{\eps} &= -\sum_{j,k,l=1}^n\chi^{jkl}\left(\frac{\cdot}{\eps}\right)\partial^3_{jkl} u&  &\text{on }\partial\Omega.\end{aligned}\right.
\end{align}
As we have done for the boundary corrector $\theta^{\eps}$, we can transform the problem \eqref{thetaepschi problem} into the divergence-form problem
\begin{align*}
\left\{\begin{aligned}-\nabla\cdot \left(A^{\mathrm{div}}\left(\frac{\cdot}{\eps}\right)\nabla\theta^{\eps}_{\chi}\right) &= 0& &\text{in }\Omega,\\
\theta^{\eps}_{\chi} &= -\sum_{j,k,l=1}^n\chi^{jkl}\left(\frac{\cdot}{\eps}\right)\partial^3_{jkl} u&  &\text{on }\partial\Omega,\end{aligned}\right.
\end{align*}
with a coefficient $A^{\mathrm{div}}\in C^{0,\alpha}(\mathbbm{T}^n;\R^{n\times n})$ for some $\alpha\in (0,1]$ that is uniformly elliptic. Let us note that since $u\in W^{5,q}(\Omega)$ for some $q>n$ and $\chi^{jkl}\in C^{2,\beta}(\R^n)$ for some $\beta\in (0,1]$ by elliptic regularity theory \cite{GT01}, we can apply Lemma \ref{W1p est for divform} to find the bound
\begin{align}\label{corrector thetaepschi est}
\eps\|\theta_{\chi}^{\eps}\|_{W^{1,p}(\Omega)}\leq C\eps \sum_{j,k,l=1}^n \left\|\chi^{jkl}\left(\frac{\cdot}{\eps}\right)\partial^3_{jkl} u\right\|_{W^{1,\infty}(\Omega)}=\calO(1)
\end{align}
for any $p\in (1,\infty)$. Finally, we introduce the function $w^{\eps}$ to be the solution to the problem
\begin{align}\label{weps problem}
\left\{\begin{aligned}-A\left(\frac{\cdot}{\eps}\right):D^2 w^{\eps} &= \sum_{i,j,k,l=1}^n a_{ij}\left(\frac{\cdot}{\eps}\right) \partial_i v^{kl}\left(\frac{\cdot}{\eps}\right)\partial^3_{jkl} u & &\text{in }\Omega,\\
w^{\eps} &= 0&  &\text{on }\partial\Omega.\end{aligned}\right.
\end{align}
Now we are in a position to prove the main result.

\begin{proof}[Proof of Theorem \ref{Main theorem}]
Let $\phi^{\eps}$ be given by \eqref{phieps} and $w^{\eps}$ be the solution to \eqref{weps problem}. Then we have that the function $u^{\eps}-u-\phi^{\eps}-2\eps w^{\eps}$ satisfies the problem
\begin{align*}
\left\{\begin{aligned}-A\left(\frac{\cdot}{\eps}\right):D^2 (u^{\eps}-u-\phi^{\eps}-2\eps w^{\eps})&=\eps^2 F^{\eps}_3 & &\text{in }\Omega,\\
u^{\eps}-u-\phi^{\eps}-2\eps w^{\eps} &= 0&  &\text{on }\partial\Omega,\end{aligned}\right.
\end{align*}
with $F^{\eps}_3$ given by
\begin{align*}
F^{\eps}_3:=\sum_{i,j,k,l=1}^n a_{ij}\left(\frac{\cdot}{\eps}\right)v^{kl}\left(\frac{\cdot}{\eps}\right)\partial^4_{ijkl}u.
\end{align*}
As $u\in W^{5,q}(\Omega)$ for $q>n$, we have that $F^{\eps}_3$ is uniformly bounded in $L^q(\Omega)$ and hence, by the uniform estimate from Lemma \ref{C1alpha est}, we find 
\begin{align}\label{est1 to combine}
\|u^{\eps}-u-\phi^{\eps}-2\eps w^{\eps}\|_{W^{1,\infty}(\Omega)}\leq C \eps^2 \|F^{\eps}_3\|_{L^q(\Omega)}=\calO(\eps^2).
\end{align}
Now, let us define the function
\begin{align}\label{psieps}
\psi^{\eps} :=  \eps^2\left[ \sum_{j,k,l=1}^n\chi^{jkl}\left(\frac{\cdot}{\eps}\right)\partial^3_{jkl}  u +  \theta^{\eps}_{\chi}\right]
\end{align} 
with $\chi^{jkl}$ given by \eqref{chijkl} and $\theta^{\eps}_{\chi}$ given by \eqref{thetaepschi problem}. Then we have that the function $w^{\eps}+z^{\eps}-\psi^{\eps}$ satisfies the problem
\begin{align*}
\left\{\begin{aligned}-A\left(\frac{\cdot}{\eps}\right):D^2 (w^{\eps}+z^{\eps}-\psi^{\eps})&=\eps F^{\eps}_4 & &\text{in }\Omega,\\
w^{\eps}+z^{\eps}-\psi^{\eps} &= 0&  &\text{on }\partial\Omega,\end{aligned}\right.
\end{align*}
with $F^{\eps}_4$ given by
\begin{align*}
F^{\eps}_4:=\sum_{d,i,j,k,l=1}^n a_{ij}\left(\frac{\cdot}{\eps}\right) \left[2\, \partial_i \chi^{dkl}\left(\frac{\cdot}{\eps}\right)\partial^4_{djkl}u+\eps \chi^{dkl}\left(\frac{\cdot}{\eps}\right)\partial^5_{dijkl}u\right].
\end{align*}
As $u\in W^{5,q}(\Omega)$ for some $q>n$, we have that $F^{\eps}_4$ is uniformly bounded in $L^q(\Omega)$ and hence, by the uniform estimate from Lemma \ref{C1alpha est}, we find
\begin{align}\label{est2 to combine}
\|w^{\eps}+z^{\eps}-\psi^{\eps}\|_{W^{1,\infty}(\Omega)}\leq  C \eps \|F^{\eps}_4\|_{L^q(\Omega)}=\calO(\eps).
\end{align}
Combining the bounds \eqref{est1 to combine} and \eqref{est2 to combine}, we obtain
\begin{align*}
\|u^{\eps}-u+2\eps z^{\eps}-\phi^{\eps}-2\eps \psi^{\eps}\|_{W^{1,\infty}(\Omega)}=\calO(\eps^2),
\end{align*}
and therefore, using the definitions of $\phi^{\eps}$ and $\psi^{\eps}$ from \eqref{phieps} and \eqref{psieps}, we have that
\begin{align}\label{a W1infty estimate}
\left\|u^{\eps}-u+2\eps z^{\eps}-\eps^2 V\left(\frac{\cdot}{\eps}\right):D^2 u-\eps^2\theta^{\eps}-2\eps^3 \theta^{\eps}_{\chi}\right\|_{W^{1,\infty}(\Omega)}=\calO(\eps^2).
\end{align}
Finally, using the rate of convergence of $z^{\eps}$ to $z$ given by \eqref{zeps to z rate}, and Lemma \ref{lemma boundary corrector} and the estimate \eqref{corrector thetaepschi est} to bound the boundary correctors, we conclude that
\begin{align*}
\left\|u^{\eps}-u+2\eps z -\eps^2 V\left(\frac{\cdot}{\eps}\right):D^2 u\right\|_{W^{1,p}(\Omega)}=\calO(\eps^{1+\frac{1}{p}})
\end{align*}
for any $p\in (1,\infty)$.
\end{proof}

\section{Numerical experiments}\label{Sec Num exp}

\subsection{An explicit $c$-bad matrix}\label{Sec expl cbad}

In this section, we prove that the matrix-valued function $A:\R^2\rightarrow \R^{2\times 2}$ given by
\begin{align}\label{cbad A}
A(y):=\frac{1}{r(y)}\begin{pmatrix}
1-\frac{1}{2}\sin(2\pi y_1)\sin(2\pi y_2) & 0\\0 & 1+\frac{1}{2}\sin(2\pi y_1)\sin(2\pi y_2)
\end{pmatrix}
\end{align}
with $r:\R^2\rightarrow\R$ defined by
\begin{align}\label{cbad r}
r(y):= 1+\frac{1}{4}(\cos(2\pi y_1)-2\sin(2\pi y_1))\sin(2\pi y_2)
\end{align}
is $c$-bad (recall the notion of $c$-bad from Corollary \ref{Cor c-good}). We observe the following:

\begin{remark} The function $r:\R^2\rightarrow\R$ given by \eqref{cbad r} is the invariant measure of $A:\R^2\rightarrow \R^{2\times 2}$ given by \eqref{cbad A}. Further note that the problem \eqref{ueps problem} can then be transformed into the divergence-form problem
\begin{align}\label{divform for numerics}
\left\{\begin{aligned}-\nabla\cdot \left(A^{\div}\left(\frac{\cdot}{\eps}\right)\nabla u^{\eps}  \right)&=r\left(\frac{\cdot}{\eps}\right) f& &\text{in }\Omega,\\
u^{\eps}&=0&  &\text{on }\partial\Omega,\end{aligned}\right.
\end{align}
with the matrix-valued function $A^{\div}:\R^n\rightarrow\R^{n\times n}$ given by
\begin{align*}
A^{\div}(y):=\begin{pmatrix}
1-\frac{1}{2}\sin(2\pi y_1)\sin(2\pi y_2) & \frac{1}{2}\cos(2\pi y_1)\cos(2\pi y_2)\\ -\frac{1}{2}\cos(2\pi y_1)\cos(2\pi y_2) & 1+\frac{1}{2}\sin(2\pi y_1)\sin(2\pi y_2)
\end{pmatrix}.
\end{align*}
\end{remark}

We can check that $A$ is $c$-bad by explicitly computing the matrix of corrector functions $V=(v^{ij})_{1\leq i,j\leq 2}$ given by \eqref{vij problem} and computing the values $\{c_j^{kl}\}_{1\leq j,k,l\leq 2}$ given by \eqref{cjkl def}.

\begin{proof}[Proof of Theorem \ref{thm explicit cbad}]
The effective coefficient $\bar{A}\in \calS^2_{+}$ is given by
\begin{align}\label{eff coeff}
\bar{A}=\int_Y Ar = \begin{pmatrix}
1 & 0\\ 0 & 1
\end{pmatrix}
\end{align} 
and it is a straightforward calculation to check that the matrix of corrector functions $V=(v^{ij})_{1\leq 1,j\leq 2}:\R^2\rightarrow\R^{2\times 2}$ is given by 
\begin{align*}
V(y)=-\frac{\sin(2\pi y_2)}{32\pi^2}\begin{pmatrix}
\cos(2\pi y_1) & 0\\0 &  \cos(2\pi y_1)-4\sin(2\pi y_1)
\end{pmatrix}.
\end{align*}
Computation of the values $c_j^{kl}$ for $j,k,l\in\{1,2\}$ given by \eqref{cjkl def} yields that
\begin{align*}
c_1^{11}&=\int_Y r a_{11}\partial_1 v^{11} =-\frac{1}{128\pi}=\int_Y r a_{11}\partial_{1} v^{22}=c_1^{22}
\end{align*}
for the values of $c_1^{11},c_1^{22}$, and that
\begin{align*}
c_2^{11}&=\int_Y r a_{22}\partial_{2} v^{11} =0=\int_Y r a_{22}\partial_{2} v^{22}=c_2^{22}
\end{align*}
for the values of $c_2^{11},c_2^{22}$. Clearly we have that $c_j^{kl}=0$ for any $(j,k,l)\in \{1,2\}^3$ with $k\neq l$.
\end{proof}
Let us note that the effective coefficient \eqref{eff coeff} is the identity matrix and hence, the homogenized problem for this $c$-bad matrix is the Poisson problem
\begin{align}\label{poisson}
\left\{\begin{aligned}-\Delta u &= f&  &\text{in }\Omega,\\
u &= 0&  &\text{on }\partial\Omega.\end{aligned}\right.
\end{align}
Further, we have that the function $z$ defined by \eqref{z problem} is given as the solution to the Poisson problem
\begin{align}\label{z for expl cbad}
\left\{\begin{aligned}-\Delta z &= -\frac{\partial_{1}f}{128\pi}&  &\text{in }\Omega,\\
z &= 0&  &\text{on }\partial\Omega.\end{aligned}\right.
\end{align}

Finally, let us note that the factor $\frac{1}{r}$ in the definition of the $c$-bad matrix \eqref{cbad A} is crucial for $c$-badness. Indeed, removing this factor we obtain a $c$-good matrix: 

\begin{remark}\label{cgood A}
The matrix-valued function $A:\R^2\rightarrow \R^{2\times 2}$ given by
\begin{align*}
A(y):=\begin{pmatrix}
1-\frac{1}{2}\sin(2\pi y_1)\sin(2\pi y_2) & 0\\0 & 1+\frac{1}{2}\sin(2\pi y_1)\sin(2\pi y_2)
\end{pmatrix}
\end{align*}
is $c$-good.
\end{remark} 

\begin{proof}
The invariant measure is the constant function $r\equiv 1$ and hence, the effective coefficient $\bar{A}\in \calS^2_{+}$ is given by
\begin{align*}
\bar{A}=\int_Y A = \begin{pmatrix}
1 & 0\\ 0 & 1
\end{pmatrix}.
\end{align*} 
It is a straightforward calculation to check that the matrix of corrector functions $V=(v^{ij})_{1\leq 1,j\leq 2}:\R^2\rightarrow\R^{2\times 2}$ is given by 
\begin{align*}
V(y)=-\frac{\sin(2\pi y_1)\sin(2\pi y_2)}{16\pi^2}\begin{pmatrix}
1 & 0\\0 & -1
\end{pmatrix}.
\end{align*}
Computation of the values $c_j^{kl}$ given by \eqref{cjkl def} yields $c_j^{kl}=0$ for all $j,k,l\in\{1,2\}$.
\end{proof}

Note that the effective problem for this $c$-good matrix is again the Poisson problem \eqref{poisson}, i.e., the homogenized solution coincides with the one from the $c$-bad problem.

\subsection{Numerical illustration of the $L^{\infty}(\Omega)$ rates}\label{Sec Num ill 1}

We consider the problem \eqref{ueps problem} with the c-bad coefficient matrix $A$ from Theorem \ref{thm explicit cbad}, the domain $\Omega:=(0,1)^2$ and the right-hand side
\begin{align*}
f:\bar{\Omega}\rightarrow \R,\quad f(x_1,x_2):=8\pi^2 \sin(2\pi x_1)\sin(2\pi x_2).
\end{align*}
Then, the solution to the homogenized problem \eqref{poisson} is given by
\begin{align*}
u:\bar{\Omega}\rightarrow\R,\quad u(x_1,x_2)=\sin(2\pi x_1)\sin(2\pi x_2),
\end{align*}
and the solution $z$ to the problem \eqref{z for expl cbad} is given by
\begin{align*}
z:\bar{\Omega}\rightarrow\R,\quad z(x)=\frac{1}{64}\left( \frac{\cosh(2\pi x_1-\pi)}{\cosh(\pi)}-\cos(2\pi x_1)\right)\sin(2\pi x_2).
\end{align*}
Figure \ref{fig 1} illustrates the estimate \eqref{Guo-Tran-Yu} from Remark \ref{rmrk some estimates}, i.e., for several values of $\eps$, we plot
\begin{align}\label{E0infty def}
E_{0,\infty}^{\eps}:=\|u^{\eps}-u+2\eps z\|_{L^{\infty}(\Omega)}
\end{align}
We approximate the solution $u^{\eps}$ to \eqref{ueps problem} with $\mathbbm{P}_1$ finite elements on a fine mesh, based on the natural variational formulation of the divergence-form problem \eqref{divform for numerics}. We observe the rate $E_{0,\infty}^{\eps}=\calO(\eps^2)$ as $\eps$ tends to zero, as expected from Remark \ref{rmrk some estimates}.

\begin{figure}[H]
\centering
\includegraphics[scale=1]{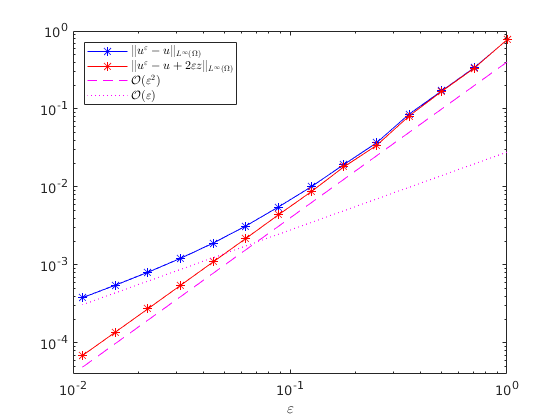}
\caption{\textit{blue}: Plot of $\|u^{\eps}-u\|_{L^{\infty}(\Omega)}$, \textit{red}: Plot of $E_{0,\infty}^{\eps}$ (see \eqref{E0infty def}). We observe $\|u^{\eps}-u\|_{L^{\infty}(\Omega)}=\calO(\eps)$ and $E_{0,\infty}^{\eps}=\calO(\eps^2)$ as expected from Remark \ref{rmrk some estimates}.}
\label{fig 1}
\end{figure}

\subsection{Numerical illustration of the $W^{1,p}(\Omega)$ rates}\label{Sec Num ill 2}

We consider the problem \eqref{ueps problem} with the c-bad coefficient matrix $A$ from Theorem \ref{thm explicit cbad}, the domain $\Omega:=(0,1)^2$ and the right-hand side
\begin{align}\label{f for num}
f:\bar{\Omega}\rightarrow \R,\quad f(x_1,x_2):=x_1(1-x_1)+x_2(1-x_2).
\end{align}
Then, the solution of the homogenized problem \eqref{poisson} is given by
\begin{align}\label{u for num}
u:\bar{\Omega}\rightarrow\R,\quad u(x_1,x_2)=\frac{1}{2}x_1(1-x_1)x_2(1-x_2).
\end{align}
Figure \ref{fig 2} illustrates the estimate \eqref{gradestimate} from Remark \ref{rmrk some estimates}, i.e., for several values of $\eps$, we plot
\begin{align}\label{E1p def}
E_{1,p}^{\eps}:=\left\|\nabla u^{\eps}-\nabla u +2\eps\nabla z - \eps \sum_{i,j=1}^n \nabla v^{ij}\left(\frac{\cdot}{\eps}\right)\partial^2_{ij} u  \right\|_{L^{p}(\Omega)}
\end{align}
for the values $p=2,3,4,5$. We approximate the solution $u^{\eps}$ to \eqref{ueps problem} and the solution $z$ to \eqref{z for expl cbad} with $\mathbbm{P}_2$ finite elements on a fine mesh, based on the natural variational formulation of the divergence-form problems \eqref{divform for numerics} and \eqref{z for expl cbad}. We observe the rate $E_{1,p}^{\eps}=\calO(\eps^{1+\frac{1}{p}})$ as $\eps$ tends to zero, as expected from Remark \ref{rmrk some estimates}.

\begin{figure}[H]
\centering
\includegraphics[scale=1]{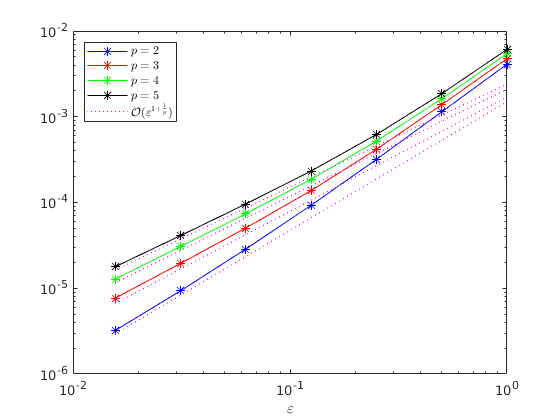}
\caption{Plot of $E_{1,p}^{\eps}$ (see \eqref{E1p def}) for $p=2,3,4,5$. We observe the rate $E_{1,p}^{\eps}=\calO(\eps^{1+\frac{1}{p}})$ as expected from Remark \ref{rmrk some estimates}.}
\label{fig 2}
\end{figure}

\subsection{Numerical illustration of the $W^{2,p}(\Omega)$ rates}\label{Sec Num ill 3}

We consider the problem \eqref{ueps problem} with the c-bad coefficient matrix $A$ from Theorem \ref{thm explicit cbad}, the domain $\Omega:=(0,1)^2$ and $f$ given by \eqref{f for num}. As before, the homogenized solution is given by \eqref{u for num}. Figure \ref{fig 3} illustrates the estimate \eqref{hessianestimate} from Remark \ref{rmrk some estimates}, i.e., for several values of $\eps$, we plot
\begin{align}\label{E2p def}
E_{2,p}^{\eps}:=\left\|D^2 u^{\eps}- D^2 u -\sum_{i,j=1}^n  D^2 v^{ij}\left(\frac{\cdot}{\eps}\right)\partial^2_{ij}u  \right\|_{L^p(\Omega)}
\end{align}
for the values $p=2,3,4,5$. We approximate the solution $u^{\eps}$ to \eqref{ueps problem} with an $H^2$ conforming finite element method on a fine mesh, using the HCT element in FreeFem\texttt{++} \cite{Hec12}. We multiply the equation \eqref{ueps problem} by the invariant measure and use the variational formulation from the framework of linear nondivergence-form equations with Cordes coefficients (see \cite{SS13}): The solution $u^{\eps}$ to \eqref{ueps problem} is the unique function in $H:=H^2(\Omega)\cap H^1_0(\Omega)$ such that there holds
\begin{align*}
\int_{\Omega} \frac{\mathrm{tr}\left(\left[rA\right]\left(\frac{\cdot}{\eps}\right)\right)}{\left\lvert \left[rA\right]\left(\frac{\cdot}{\eps}\right)\right\rvert^2} \left(-\left[rA\right]\left(\frac{\cdot}{\eps}\right):D^2 u^{\eps}\right)\Delta v = \int_{\Omega}\frac{\mathrm{tr}\left(\left[rA\right]\left(\frac{\cdot}{\eps}\right)\right)}{\left\lvert \left[rA\right]\left(\frac{\cdot}{\eps}\right)\right\rvert^2}\, r\left(\frac{\cdot}{\eps}\right)f\,\Delta v
\end{align*}
for any $v\in H$. We observe the rate $E_{2,p}^{\eps}=\calO(\eps^{\frac{1}{p}})$ as $\eps$ tends to zero, as expected from Remark \ref{rmrk some estimates}.
  
\begin{figure}[H]
\centering
\includegraphics[scale=1]{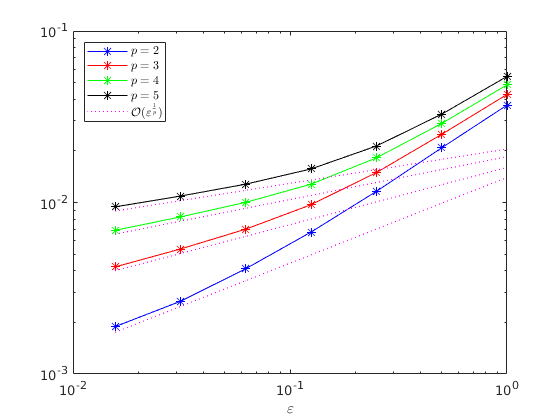}
\caption{Plot of $E_{2,p}^{\eps}$ (see \eqref{E2p def}) for $p=2,3,4,5$. We observe the rate $E_{2,p}^{\eps}=\calO(\eps^{\frac{1}{p}})$ as expected from Remark \ref{rmrk some estimates}.}
\label{fig 3}
\end{figure}

\subsection{Comparison of $c$-bad and $c$-good problems}\label{Sec Num ill 4}

We refer to the problem \eqref{ueps problem} with the $c$-bad coefficient matrix from Theorem \ref{thm explicit cbad} as the $c$-bad problem and to the problem \eqref{ueps problem} with the $c$-good coefficient matrix from Remark \ref{cgood A} as the $c$-good problem. We perform experiments for these two problems with two different choices of right-hand sides, one with known homogenized solution $u$ and one with unknown homogenized solution $u$. All experiments are performed on the domain $\Omega:=(0,1)^2$.

Let us recall that the homogenized problems corresponding to the $c$-bad and the $c$-good problem coincide and that the homogenized solution $u$ is the solution to the Poisson problem \eqref{poisson}.

\subsubsection{$c$-bad and $c$-good problems with known (common) homogenized function $u$}

We consider the right-hand side $f$ given by \eqref{f for num}. Then, the solution $u$ of the homogenized problem is known and given by \eqref{u for num}. 

Figure \ref{fig 4} illustrates the $L^{\infty}$ convergence rate $\calO(\eps)$ for the $c$-bad problem and the convergence rate $\calO(\eps^2)$ for the $c$-good problem. We also illustrate the corrected $L^{\infty}$ bound $E_{0,\infty}^{\eps}=\calO(\eps^2)$ for the $c$-bad problem. We approximate the solution $u^{\eps}$ to \eqref{ueps problem} and the solution $z$ to \eqref{z for expl cbad} with $\mathbbm{P}_2$ finite elements on a fine mesh, based on the natural variational formulation of the divergence-form problems \eqref{divform for numerics} (note $r\equiv 1$ for the $c$-good problem) and \eqref{z for expl cbad}.

\begin{figure}[H]
\setlength{\abovecaptionskip}{17.5 pt plus 4pt minus 2pt}
\mbox{
\subfigure{\includegraphics[scale=.49]{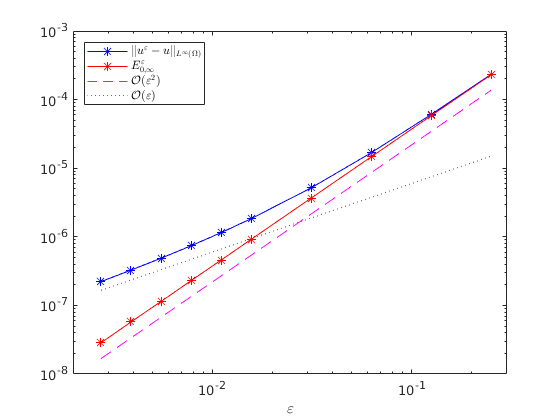}}\,
\subfigure{\includegraphics[scale=.49]{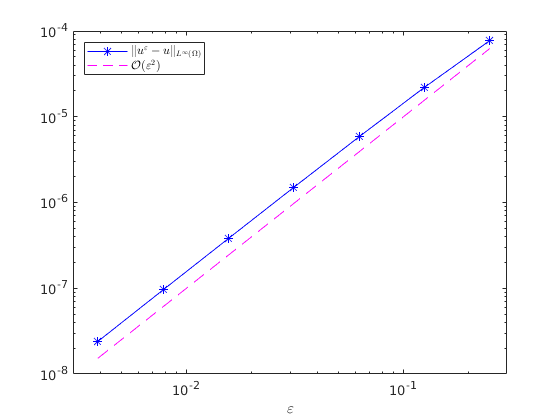}}
}
\vspace{-0.6cm}
\caption{Illustration of the $L^{\infty}$-rates $\|u^{\eps}-u\|_{L^{\infty}(\Omega)}=\calO(\eps)$ and $E_{0,\infty}^{\eps}=\calO(\eps^2)$ for the $c$-bad problem (left), and $\|u^{\eps}-u\|_{L^{\infty}(\Omega)}=\calO(\eps^2)$ for the $c$-good problem (right) with the right-hand side \eqref{f for num}.}
\label{fig 4}
\end{figure}

\subsubsection{$c$-bad and $c$-good problems with unknown (common) homogenized function $u$}

We consider the right-hand side $f$ given by
\begin{align}\label{f for num 2}
f:\bar{\Omega}\rightarrow \R,\quad f(x):=x_1^3 (1-x_1)^3 \sin(2\pi (x_1-2x_2)).
\end{align}
Let us note that we do not know the homogenized solution $u$ exactly, we have however that $u\in H^6(\Omega)\cap H^1_0(\Omega)$ as the right-hand side $f\in H^4(\Omega)$ satisfies the compatibility conditions $f=0$ and $\partial_1^2 f - \partial_2^2 f = 0$ at the corners of the square $(0,1)^2=\Omega$; see \cite{HO14}.

Figure \ref{fig 5} illustrates the $L^{\infty}$ convergence rate $\calO(\eps)$ for the $c$-bad problem and the convergence rate $\calO(\eps^2)$ for the $c$-good problem. We also illustrate the corrected $L^{\infty}$ bound $E_{0,\infty}^{\eps}=\calO(\eps^2)$ for the $c$-bad problem. We approximate the functions $u^{\eps}$, $u$ and $z$ with $\mathbbm{P}_2$ finite elements as before.

\begin{figure}[H]
\setlength{\abovecaptionskip}{17.5 pt plus 4pt minus 2pt}
\mbox{
\subfigure{\includegraphics[scale=.49]{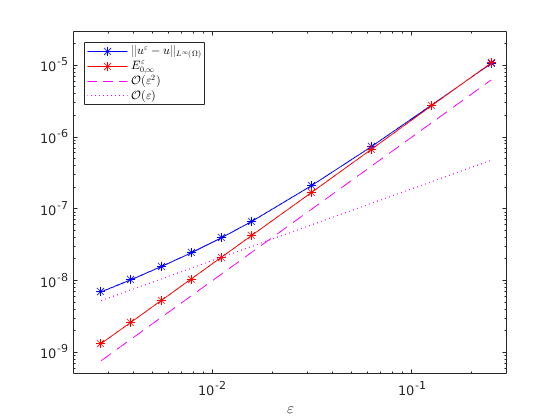}}\,
\subfigure{\includegraphics[scale=.49]{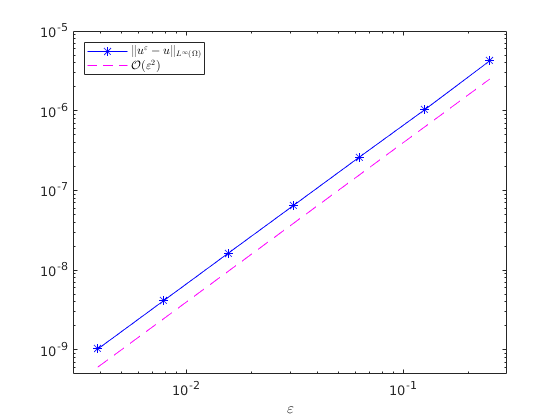}}
}
\vspace{-0.6cm}
\caption{Illustration of the $L^{\infty}$-rates $\|u^{\eps}-u\|_{L^{\infty}(\Omega)}=\calO(\eps)$ and $E_{0,\infty}^{\eps}=\calO(\eps^2)$ for the $c$-bad problem (left), and $\|u^{\eps}-u\|_{L^{\infty}(\Omega)}=\calO(\eps^2)$ for the $c$-good problem (right) with the right-hand side \eqref{f for num 2}.}
\label{fig 5}
\end{figure}

\section{Extensions and concluding remarks}\label{Sec Extensions and concl}

\subsection{Nonsmooth domains}

The smoothness assumption on the domain $\Omega$ is used to deduce regularity of $u$ from the regularity assumption on $f$, and it ensures that the uniform estimates from Section \ref{subsec uniform est} hold. We briefly discuss extensions to nonsmooth domains.

\subsubsection{$C^{2,\gamma}$ domains}
We note that the uniform estimates from Section \ref{subsec uniform est} still hold for bounded $C^{2,\gamma}$ domains. Therefore, the result of Theorem \ref{Main theorem} remains valid with the additional assumption $u\in W^{5,q}(\Omega)$ for some $q>n$, which has previously been deduced from the regularity of the right-hand side $f$. 

\subsubsection{Convex domains}
We would like to briefly discuss the case of convex domains. Let $\Omega\subset\R^n$ be a bounded convex domain in dimension $n\geq 2$ and assume that the homogenized solution is of regularity $u\in W^{5,q}(\Omega)$ for some $q>n$. Let us further assume that the coefficient is of regularity $A\in C^{0,\alpha}(\mathbbm{T}^n;\calS^n_{+})$ for some $\alpha\in (0,1]$ and satisfies the Cordes condition (which dates back to \cite{Cor56}), i.e., that there exists a constant $\delta\in (0,1]$ such that there holds
\begin{align}\label{Cordes}
\frac{\lvert A\rvert^2}{(\mathrm{tr} A)^2}\leq \frac{1}{n-1+\delta}\quad\text{in }\R^n.
\end{align}
Let us note that the Cordes condition \eqref{Cordes} is a consequence of uniform ellipticity in two dimensions, i.e., \eqref{Cordes} holds for any $A\in C^{0,\alpha}(\mathbbm{T}^2;\calS^2_{+})$. Let us also note that Theorem \ref{Capdeboscq-Sprekeler-Suli} holds in this situation for $p=2$; see \cite{CSS20}.

In the situation described above, there exists a unique solution $u^{\eps}\in H^2(\Omega)\cap H^1_0(\Omega)$ to \eqref{ueps problem} and we have a uniform $H^2$ estimate \cite[Theorem 2.5]{CSS20}. Therefore, by the Sobolev embedding, we have the uniform $W^{1,p}$ estimate 
\begin{align*}
\|u^{\eps}\|_{W^{1,p}(\Omega)}\leq C \|u^{\eps}\|_{H^2(\Omega)}\leq C  \|f\|_{L^2(\Omega)}
\end{align*}
for any $p<2^*$ with constants independent of $\eps$. Here, we write $2^*:=\frac{2n}{n-2}$ to denote the critical Sobolev exponent (with the convention that $2^*:=\infty$ if $n=2$). This uniform estimate replaces the need for the uniform $C^{1,\alpha}$ estimate from Lemma \ref{C1alpha est}.

Finally, in order to estimate the boundary corrector, we transformed the problem \eqref{thetaeps problem} into divergence-form and used that for problems of the form 
\begin{align*}
\left\{\begin{aligned}-\nabla\cdot \left(A^{\mathrm{div}}\left(\frac{\cdot}{\eps}\right)\nabla \rho^{\eps}\right) &= -\nabla\cdot F& &\text{in }\Omega,\\
\rho^{\eps} &= g&  &\text{on }\partial\Omega,\end{aligned}\right.
\end{align*}
we have (Lemma \ref{W1p est for divform}) the uniform $W^{1,p}$ estimate 
\begin{align}\label{uniform w1p smooth}
\|\rho^{\eps}\|_{W^{1,p}(\Omega)}\leq C\left(\|F\|_{L^p(\Omega)}+\|g\|_{W^{1,p}(\Omega)}\right)
\end{align}
with a constant $C>0$ independent of $\eps$, assuming that $A^{\mathrm{div}}\in C^{0,\alpha}(\mathbbm{T}^n;\R^{n\times n})$ for some $\alpha\in (0,1]$ is uniformly elliptic and that $\Omega$ is sufficiently smooth. 

Now as $\Omega$ is merely assumed to be convex, we still have \eqref{uniform w1p smooth} for $p=2$ by standard arguments and hence, we find that the result of Theorem \ref{Main theorem} remains true for $p=2$ under the assumptions made in this section. Uniform $W^{1,p}$ estimates for divergence-form problems for a wider range of values $p$ require a more sophisticated approach. With a symmetry assumption on $A^{\mathrm{div}}$, uniform $W^{1,p}$ estimates for divergence-form problems on Lipschitz domains (recall that bounded convex domains are Lipschitz \cite{Gri11}) have been obtained in \cite{She08} for values of $p$ in a certain range around $p=2$.

\subsection{Interpolation}\label{Sec: interpolation}

Let us revisit Remark \ref{rmrk some estimates} and note that the gradient bound \eqref{gradestimate} follows from the $L^{\infty}$ bound \eqref{Guo-Tran-Yu} and the Hessian bound \eqref{hessianestimate} via the Gagliardo--Nirenberg interpolation inequality \cite{Nir59} applied to the function
\begin{align*}
\fhi^{\eps}:=u^{\eps}- u +2\eps z - \eps^2\,  V\left(\frac{\cdot}{\eps}\right):D^2 u .
\end{align*}
Indeed, let us assume that $\|\fhi^{\eps}\|_{L^{\infty}(\Omega)}=\calO(\eps^2)$ and $\|D^2 \fhi^{\eps}\|_{L^p(\Omega)}=\calO(\eps^{\frac{1}{p}})$ for any $p\in (1,\infty)$. Then the Gagliardo--Nirenberg inequality yields
\begin{align*}
\|\nabla \fhi^{\eps}\|_{L^p(\Omega)}\leq C \left( \|D^2 \fhi^{\eps}\|_{L^{\frac{p}{2}}(\Omega)}^{\frac{1}{2}} \|\fhi^{\eps}\|_{L^{\infty}(\Omega)}^{\frac{1}{2}} + \|\fhi^{\eps}\|_{L^{\infty}(\Omega)}\right) = \calO(\eps^{1+\frac{1}{p}})
\end{align*}
for any $p\in (2,\infty)$. 
This shows once again that the optimality of the bounds \eqref{Guo-Tran-Yu}--\eqref{hessianestimate} is natural.
We conclude this paper with a review of the main results.

\subsection{Conclusion}

In this paper we derived optimal rates of convergence in the periodic homogenization of linear elliptic equations in nondivergence-form. 
As a result of a $W^{1,p}$ corrector estimate, we obtained that the optimal rate of convergence of $u^{\eps}$ to the homogenized solution in the $W^{1,p}$-norm is $\calO(\eps)$ and also recovered that the optimal convergence rate in the $L^{\infty}$-norm is $\calO(\eps)$. 
Moreover, we obtained optimal estimates for the gradient and the Hessian of the solution with correction terms taken into account in $L^p$-norm. 

In the final part of the paper, we provided an example of an explicit $c$-bad matrix and presented several numerical experiments matching the theoretical results and illustrating the optimality of the obtained rates.

\addtocontents{toc}{\protect\setcounter{tocdepth}{0}}

\section*{Acknowledgements}

The authors thank Professor Nam Le (Indiana University Bloomington) for the suggestion of adding Section \ref{Sec: interpolation}. TS is supported by the UK Engineering and Physical Sciences Research Council [EP/L015811/1]. HT is supported in part by NSF grant DMS-1664424 and NSF CAREER grant DMS-1843320.

\bibliographystyle{plain}
\bibliography{ref}

\end{document}